\documentclass[12pt,twoside,reqno]{amsart}

\usepackage[OT1]{fontenc}    %
\usepackage{type1cm}         %
\usepackage{amsthm}          %
\usepackage[dvips]{graphicx} 
\usepackage[bf]{caption}     
\usepackage{psfrag}          

\usepackage{verbatim}        

\numberwithin{equation}{section}

\theoremstyle{plain}
\newtheorem{theorem}{Theorem}[section]
\newtheorem{corollary}[theorem]{Corollary}
\newtheorem{proposition}[theorem]{Proposition}
\newtheorem{lemma}[theorem]{Lemma}

\theoremstyle{remark}
\newtheorem{remark}[theorem]{Remark}
\newtheorem{remarks}[theorem]{Remarks}
\newtheorem{example}[theorem]{Example}

\theoremstyle{definition}
\newtheorem{definition}[theorem]{Definition}

\newcommand{\HH}{\mathcal{H}}

\newcommand{\R}{\mathbb{R}}

\newcommand{\Q}{\mathbb{Q}}

\newcommand{\N}{\mathbb{N}}
\newcommand{\eps}{\varepsilon}
\newcommand{\roo}{\varrho}

\DeclareMathOperator{\dimm}{dim_M}
\DeclareMathOperator{\dimh}{dim_H}
\DeclareMathOperator{\dimp}{dim_p}
\DeclareMathOperator{\dist}{dist}
\DeclareMathOperator{\diam}{diam}

\DeclareMathOperator{\por}{por}

\DeclareMathOperator{\spt}{spt}

\begin{document}

\title[Packing dimension and Ahlfors regularity of porous sets]
{Packing dimension and Ahlfors regularity of porous sets in metric spaces}

\author[ J\"arvenp\"a\"a     \and
         J\"arvenp\"a\"a     \and
         K\"aenm\"aki        \and
         Rajala              \and
         Rogovin             \and
         Suomala                  ]
       { E. J\"arvenp\"a\"a     \and
         M. J\"arvenp\"a\"a  \and
         A. K\"aenm\"aki      \and
         T. Rajala            \and
         S. Rogovin            \and
         V. Suomala                }

\address{Department of Mathematics and Statistics \\
         P.O. Box 35 (MaD) \\
         FI-40014 University of Jyv\"askyl\"a \\
         Finland}

\email{esaj@maths.jyu.fi}
\email{amj@maths.jyu.fi}
\email{antakae@maths.jyu.fi}
\email{tamaraja@maths.jyu.fi}
\email{sakallun@maths.jyu.fi}
\email{visuomal@maths.jyu.fi}

\thanks{EJ, MJ, AK, TR and SR acknowledge the support of the Academy of
Finland, projects \#211229, \#114821 and \#212753. TR was also partially
supported by the Vilho, Yrj\"o and Kalle V\"ais\"al\"a fund.}
\date{\today}

\begin{abstract}
Let $X$ be a metric measure space with an $s$-regular measure $\mu$.
We prove that if $A\subset X$ is $\varrho$-porous, then
$\dimp(A)\le s-c\varrho^s$ where $\dimp$ is the packing dimension and
$c$ is a positive constant which depends on $s$ and the structure constants of
$\mu$.
This is an analogue of a well known asymptotically sharp result in Euclidean
spaces. We illustrate by an example that the corresponding result is not valid
if $\mu$ is a doubling measure. However, in the doubling case we find a fixed
$N\subset X$ with $\mu(N)=0$ such that
$\dimp(A)\le\dimp(X)-c(\log\tfrac 1\varrho)^{-1}\varrho^t$
for all $\varrho$-porous sets $A\subset X\setminus N$. Here $c$ and $t$ are
constants which depend on the structure constant of $\mu$. Finally, we
characterize
uniformly porous sets in complete $s$-regular metric spaces in terms of regular
sets by verifying that $A$ is uniformly porous if and only if there is $t<s$
and a $t$-regular set $F$ such that $A\subset F$.
\end{abstract}

\maketitle

\section{Introduction}

The purpose of this paper is twofold: we study dimensional properties of porous
sets in $s$-regular and doubling metric measure spaces and characterize
uniformly porous sets in
terms of regularity. For definitions we refer to Sections 2 and 3.

In Euclidean spaces dimensional properties of porous sets have been studied
extensively, see for example \cite{BS},
\cite{JarvenpaaJarvenpaaKaenmakiSuomala2005}, \cite{KaenmakiSuomala2004},
\cite{KoskelaRohde1997}, \cite{L}, \cite{MartioVuorinen1987},
\cite{Mattila1988}, \cite{Nieminen}, \cite{PR}, \cite{Salli1991}, \cite{T} and
references therein. It is well known that if $A\subset\mathbb R^n$ is
$\varrho$-porous, meaning that $A$ contains holes of relative size
$\varrho$ in all small balls, then
\begin{equation}\label{intro1}
\dimp(A)\le n-c\varrho^n
\end{equation}
where $\dimp$ is the packing dimension and $c$ is a positive constant
depending on $n$ only (see \cite{MartioVuorinen1987, T}). Furthermore,
\eqref{intro1} is asymptotically sharp as $\varrho$ tends to zero
(\cite{KoskelaRohde1997}, \cite[Remark 4.2]{KaenmakiSuomala2004}). In
\cite{DS} and \cite{BHR} it is shown that the dimension of a porous measure
in a (globally) $s$-regular space is smaller than $s$.
In this paper we address the question to what extend the quantitative estimate
\eqref{intro1} is valid in metric measure spaces $X$. It turns out
that the following analogue of \eqref{intro1} holds provided that
$X$ is equipped with a (locally) $s$-regular measure $\mu$: if $A\subset X$ is
$\varrho$-porous, then
\begin{equation}\label{intro2}
\dimp(A)\le s-c\varrho^s
\end{equation}
where $c$ is a positive constant which depends on $s$ and the structure
constants $a_\mu$ and $b_\mu$ of $\mu$
(see Theorem \ref{dimestsreg}). Note that by \cite[Theorem 3.16]{Cutler}
$\dimp(X)=s$ provided that $X$ is $s$-regular. We also show that the
dependence on $a_\mu$ and $b_\mu$ is necessary, that is, unlike in
$\mathbb R^n$, it is not possible to find $c$ which depends on $s$ only
(see Remark \ref{bestone?}.(3)).

In \eqref{intro2} it is not sufficient to assume that $\mu$ is doubling: in
Example \ref{infinitetree} we construct a geodesic doubling metric space $X$
having a subset with maximal dimension and porosity. However, in general
the failure of the dimension drop is due to a fixed set
with $\mu$-measure zero provided that $\mu$ is doubling. More precisely, in
Theorem \ref{dimestdoub}
we show that there exists $N\subset X$ with $\mu(N)=0$ such that
$\dimp(A)\le\dimp(X)-c(\log\tfrac 1\varrho)^{-1}\varrho^t$ for all
$\varrho$-porous
$A\subset X\setminus N$. Here $t$ and $c$ are constants which depend
on the structure constant $c_\mu$ of $\mu$.

As in Euclidean spaces, in complete $s$-regular metric measure space $X$
uniform
porosity is closely related to regularity. We prove that $A\subset X$ is
uniformly porous if and only if there are $t<s$ and a $t$-regular set
$F\subset X$ such that $A\subset F$ (see Theorem \ref{treghuok}). The easier
if-part was proven in \cite{BHR}, but we give some quantitative  estimates on
the relations between porosity, $t$ and $s$.

The paper is organized as follows: In Section 2 we discuss the concept of
porosity we are using in metric measure spaces whilst Section 3 is dedicated
to measure theoretic preliminaries. Dimension estimates for porous sets
are dealt in Section 4. In last section we focus on connections between
uniform porosity and regularity.

\section{Notation}

Let $X=(X,d)$ be a separable metric space and $A\subset X$.
For $x\in X$ and $r>0$, we set
\begin{equation}\label{lporr}
\begin{split}
  \por(A,x,r) =\sup\{\roo\ge0\,:\,&\text{there is }y\in X\text{ such that }
               B(y,\roo r)\cap A=\emptyset\\
               &\text{and }\roo r+d(x,y)\leq r\}.
\end{split}
\end{equation}
Let $B(x,r)=\{y\in X\,:\, d(y,x)\le r\}$ be the closed ball centred at $x$
with radius $r$. Since in metric spaces the centre and the radius of a ball is
not uniquely defined, we will always assume that a centre and a radius is fixed
when we use the word ball. The \emph{porosity of $A$ at a point $x$} is
defined to be
\begin{equation}\label{lpor}
  \por(A,x) = \liminf_{r \downarrow 0} \por(A,x,r)
\end{equation}
and the \emph{porosity of $A$} is given by
\begin{equation}\label{lporA}
  \por(A) = \inf_{x \in A} \por(A,x).
\end{equation}
We call  $A\subset X$ \emph{porous} if $\por(A)>0$, and more precisely,
\emph{$\roo$-porous} provided that $\por(A)\ge\roo$.
Furthermore, $A \subset X$ is
\emph{uniformly $\roo$-porous} if there exist constants $\roo>0$ and $r_p>0$
such that $\por(A,x,r)\ge\roo$ for all $x\in A$ and
$0<r<r_p$.

\begin{remarks}\label{introrems}
(1) Even though it would be more accurate to use the term lower porosity
for $\por(A,x)$ and $\por(A)$ to distinguish them from upper porosities
defined by replacing $\liminf$ by $\limsup$ in \eqref{lpor}, we keep
the terminology shorter. Upper porosities are irrelevant for our purposes;
there is no nontrivial upper
bound for dimensions of upper porous sets. In fact, there exist sets in
$\mathbb R^n$ with maximum upper porosity and with Hausdorff dimension $n$,
see \cite[\S 4.12]{Mattila1995}.

(2) We follow the convention introduced in \cite{mmpz} to use
$\por(A,x,r)$ and $\por(A,x)$ instead of
\[
\por^*(A,x,r)=\sup\{\roo\ge0\,:\,B(y,\roo r)\subset B(x,r)\setminus
A\text{ for some }y\in X\}
\]
and
\[
\por^*(A,x)=\liminf_{r\downarrow0}\por^*(A,x,r)
\]
to guarantee that $0\leq\por(A,x,r)\leq\tfrac12$ for all $A\subset X$,
$x\in A$ and $r>0$. From the point of view of our results, however,
there is no difference between $\por$ and $\por^*$ since we always
have $\por(A,x,r)\leq\por^*(A,x,r)\leq 2\por (A,x,2r)$, and therefore,
$\por(A,x)\leq\por^*(A,x)\leq 2\por(A,x)$.

(3) To emphasize the underlying metric space,
we write $\por_{(X,d)}^*$ instead of $\por^*$ in what follows.
Observe that $0 \le \por_{(\R^n,|\cdot|)}^*(A) \le \tfrac12$ for all
$A\subset \R^n$, where $|\cdot|$ denotes the usual Euclidean metric.
This is not necessarily true in general metric spaces. Indeed, choosing
$0<\eps<1$, we have $\por_{(\R^n,|\cdot|^\eps)}^*(A) =
\por_{(\R^n,|\cdot|)}^*(A)^\eps$ for every $A \subset \R^n$.
Hence, for example, $\por_{(\R^n,|\cdot|^\eps)}^*(\{ x \}) =
(\tfrac12)^\eps$ for every $x \in \R^n$. In the following remark, we show
that $*$-porosity may be exactly one.

(4)  We work in $\R^2$ with the polar coordinates.
Define
\[
X = \{(lq,2\pi q):0 \le l\le 1\text{ and } q\in\Q\cap [0,1)\}
\]
and equip
$X$ with the path metric. We claim that
  $\por^*\bigl( \{ (0,0) \} \bigr) = 1$. Let $0<r<1$. For each $i \in
  \N$ choose $q_i \in \Q \cap [0,r]$ such that $\sup \{ q_i : i \in
  \N\} = r$. It follows immediately that for every $i \in \N$ and
  $\eps > 0$
 \begin{equation*}
   B\bigl( (q_i,2\pi q_i), q_i - \eps \bigr) \subset B\bigl(
    (0,0),r \bigr) \setminus \{ (0,0) \},
  \end{equation*}
  that is, $\por^*\bigl( \{ (0,0) \}, (0,0), r \bigr) \ge
  (q_i-\eps)/r$. Hence $\por^*\bigl( \{ (0,0) \}, (0,0), r \bigr) = 1$
  for every $0<r<1$ and the claim is proved.

(5) The following simple but extremely useful fact will be
frequently needed: If $\por(A)>\roo$, then $A=\cup_{k\in\mathbb{N}}
A_k$ where
\[
A_k=\{x\in A\,:\,\por(A,x,r)>\roo\text{ for all }0<r<\frac 1k\}.
\]
Furthermore, given any $\varepsilon>0$, we may, using the separability,
write $A_k$ as a union of small pieces $A_{kj}$
such that $A_k=\cup_{j\in\mathbb N}A_{kj}$ and $\diam(A_{kj})<\varepsilon$
for all $k$ and $j$. (Here $\diam$ is the diameter of a set.)
\end{remarks}

\section{Measure theory in metric spaces}

This section contains some basic facts of measure and dimension theory
in metric spaces that will be needed later. Recall that
$X$ is a separable metric space. By a measure we
always mean a Borel regular outer measure defined on all subsets of
$X$, see \cite[Definition 1.1]{Mattila1995}. We say that $\mu$ is
$\sigma$-finite
if $X=\cup_{k\in\mathbb{N}}A_k$ where $\mu(A_k)<\infty$ for each $k$.

The separability assumption is natural given our
interest in dimension estimates since the Hausdorff dimension
of a non-separable metric space $X$ is infinite and usually one can find porous
sets $A\subset X$ that are non-separable. Moreover, no $\sigma$-finite
doubling measures exist in non-separable spaces.

We denote by $\HH^s$ the $s$-dimensional Hausdorff measure defined on $X$.
As in \cite[\S 5.3]{Mattila1995}, we define for a bounded set $A
\subset X$, $\lambda \ge 0$ and $r>0$
\begin{equation*}
  M^\lambda(A,r) = \inf\{ kr^\lambda : A \subset \bigcup_{i=1}^k
  B(x_i,r) \text{ for some } x_i\in X, k\in\mathbb N \}
\end{equation*}
with the interpretation $\inf\emptyset = \infty$.
The (upper) Minkowski dimension of a bounded set $A$ is
\begin{equation*}
  \dimm(A) = \inf\{ \lambda : \limsup_{r \downarrow 0} M^\lambda(A,r)
  < \infty \}.
\end{equation*}

The packing dimension of $A\subset X$ is given by
\[
\dimp(A)=\inf\bigl\{\sup_i\dimm(A_i)\,:\,A_i\text{ is bounded and }
A\subset\bigcup_{i=1}^{\infty}A_i\bigr\}.
\]
Alternatively, the packing dimension may be defined in terms of the
(radius based) packing measures
$\mathcal{P}^s$ (see Cutler \cite[\S 3.1]{Cutler} for the definition) by
the identity (here $\sup\emptyset=0$)
\[\dimp(A)=\sup\{s\geq 0\,:\,\mathcal{P}^s(A)>0\},\]
see \cite[Theorem 3.11]{Cutler}.
Since $\HH^\lambda(A) \le \liminf_{r\downarrow 0} M^\lambda(A,r)$ for all
bounded sets $A\subset X$, we immediately
get $\dimh(A) \le \dimp(A)\leq \dimm(A)$, where $\dimh$ denotes the Hausdorff
dimension. It is also easy to see that $\dimp(X)<\infty$ whenever $X$
carries a doubling measure, consult \cite[Theorem 3.16]{Cutler}.

Let $s>0$. A measure $\mu$ on $X$ is \emph{$s$-regular on a set $A\subset X$}
if there are constants $0<a_\mu\le b_\mu$
and $r_\mu>0$ such that
\begin{equation}\label{sreg}
  a_\mu r^s \le \mu\bigl( B(x,r) \bigr) \le b_\mu r^s
\end{equation}
for all $x\in A$ and $0<r<r_\mu$. A set $A\subset X$ is \emph{$s$-regular}
if there is a measure $\mu$ which is $s$-regular on $A$ and
$\mu(X\setminus A)=0$.
In particular, a metric space $X$ is \emph{$s$-regular} if there
is a measure $\mu$ which is $s$-regular on $X$.

A measure $\mu$ on $X$ is called
\emph{doubling} if
there are constants $c_\mu \ge 1$ and $r_\mu > 0$ such that
\begin{equation}\label{doubling}
  0<\mu\bigl( B(x,2r) \bigr) \le c_\mu \mu\bigl( B(x,r) \bigr)<\infty
\end{equation}
for every $x \in X$ and $0 < r <r_\mu$.
A metric space is \emph{doubling} if there exists a constant $N \in \N$ such
that for each $r>0$, every closed ball with radius $2r$ can be
covered by a family of at most $N$ closed balls of radius $r$.
Notice that an $s$-regular measure on $X$ is doubling, and
moreover, by \cite{LuukkainenSaksman1998}, every complete
doubling metric space carries a doubling measure.

Often in the literature it is assumed that \eqref{sreg} and \eqref{doubling}
are valid for all $0<r\le\diam(X)$, that is, $\mu$ is globally $s$-regular or
doubling (see for example \cite{BHR}). However, for our purposes this is not
needed by Remark \ref{introrems}.(5).
The following example shows that it is not always possible to choose
$r_\mu=\diam(X)$.

\begin{example}
Equip $X = [0,1] \times \mathbb{N} \subset \R^2$ with the metric $d$ defined
by
\begin{equation*}
d\bigl( (x_1,y_1),(x_2,y_2) \bigr) =
\begin{cases}
|x_1-x_2|, & y_1=y_2, \\
1,         & y_1 \ne y_2.
\end{cases}
\end{equation*}
Let $\mu=\HH^1$ be the length measure on $X$. Now $r\leq\mu(B(x,r))\leq2r$
whenever $0<r<1$, but $\mu(B(x,1))=\infty$ for all $x\in X$.
\end{example}

It is straightforward to see that in separable metric spaces all doubling
measures are $\sigma$-finite,
in particular, this is true for $s$-regular measures.

An easy exercise leads to the following lemma:

\begin{lemma} \label{thm:harjoitustehtava}
Suppose that $\mu$ is a doubling measure on $X$. For all
$x \in X$, $0<r<r_\mu$ and $\alpha>1$ we have
\begin{equation}\label{ht1}
\mu\bigl(B(x,\alpha r)\bigr)\le c_\mu^{\frac{\log\alpha}{\log 2}+1}
  \mu\bigl(B(x,r)\bigr).
\end{equation}
Moreover, if $\mu$ is $s$-regular, then
\begin{equation}\label{ht2}
\mu\bigl(B(x,\alpha r)\bigr)\le\frac{b_\mu}{a_\mu}\alpha^s\mu\bigl(B(x,r)\bigr)
\end{equation}
for all $x \in X$, $0<r<r_\mu$ and $\alpha>1$.
\end{lemma}

Let $\mu$ be an $s$-regular measure on $X$ and $A\subset X$. For all
$\lambda\ge0$ and $r>0$ we define
\begin{equation*}
 M_\mu^\lambda(A,r) = \frac{\mu\bigl( A(r) \bigr)}{r^{s-\lambda}},
\end{equation*}
where
\[
A(r) = \{ x \in X : \dist(x,A)<r \}
\]
is the open
$r$-neighbourhood of $A$ and $\dist(x,A)=\inf\{d(x,a):a\in A\}$ is the
distance of $x$ from $A$.
The following easy lemma shows how $M_\mu^\lambda(A,r)$ can be utilized to
calculate $\dimm(A)$. We give a detailed proof for the convenience of
the reader.

\begin{lemma} \label{thm:content}
Suppose that $\mu$ is an $s$-regular measure on $X$. Let $A \subset X$
and $\lambda \ge 0$. Then
\begin{equation*}
    2^{-s}b_\mu^{-1} M_\mu^\lambda(A,r) \le M^\lambda(A,r) \le 2^sa_\mu^{-1}
    M_\mu^\lambda(A,r)
\end{equation*}
whenever $0 < r < \frac{r_\mu}2$.
\end{lemma}

\begin{proof}
Fix $0<r<\frac{r_\mu}2$.
For the right-hand side inequality, we may assume that
$\mu\bigl(A(r) \bigr) < \infty$.
Since $X$ is separable, there exists a maximal collection of mutually disjoint
balls $\{ B(x_i,\frac r2) \}_{i \in I}$, where $I$ is a countable index set and
$x_i\in A$ for all $i \in I$. Observe that the maximality trivially implies
\begin{equation*}
  A \subset \bigcup_{i \in I} B(x_i,r).
\end{equation*}
Let $\# I$ be the number of elements in $I$. Since
\begin{equation*}
 \mu\bigl( A(r) \bigr) \ge \mu\biggl( \bigcup_{i \in I}
  B(x_i,\frac r2)\biggr)=\sum_{i \in I}\mu\bigl( B(x_i,\frac r2)\bigr)
  \ge\# I a_\mu 2^{-s}r^s,
\end{equation*}
it follows that $\# I < \infty$. This in turn implies that
\begin{equation*}
    M_\mu^\lambda(A,r) = \frac{\mu\bigl( A(r) \bigr)}{r^{s-\lambda}}
    \ge a_\mu 2^{-s} \# I r^\lambda \ge a_\mu 2^{-s}M^\lambda(A,r).
\end{equation*}

For the left-hand side inequality, we may assume that
$M^\lambda(A,r) < \infty$. Let $\eps>0$. Choose $k \in \N$ and
$x_1,\ldots,x_k \in X$ such that
\begin{equation*}
    A \subset \bigcup_{i=1}^k B(x_i,r) \quad \text{and} \quad
    M^\lambda(A,r) \ge kr^\lambda - \eps.
\end{equation*}
Since now
\begin{equation*}
    \mu\bigl( A(r) \bigr) \le \mu\biggl( \bigcup_{i=1}^k B(x_i,2r)
    \biggr) \le \sum_{i=1}^k \mu\bigl( B(x_i,2r) \bigr) \le kb_\mu (2r)^s,
\end{equation*}
we get
\begin{equation*}
    M_\mu^\lambda(A,r) = \frac{\mu\bigl( A(r) \bigr)}{r^{s-\lambda}}
    \le 2^sb_\mu kr^\lambda \le 2^sb_\mu\bigl( M^\lambda(A,r) + \eps \bigr).
\end{equation*}
The proof is finished by letting $\eps\downarrow 0$.
\end{proof}

The next observation shows that any porous set on a
space carrying a doubling measure must have zero measure. Note that
the proposition (with its simple proof) is easily seen to hold for
upper porous sets as well.

\begin{proposition} \label{thm:poromittanolla}
Suppose that $\mu$ is a doubling
measure on $X$. If $A \subset X$ is porous then $\mu(A) = 0$.
\end{proposition}

\begin{proof}
By Remark \ref{introrems}.(5), we may assume that $A$ is uniformly
$\roo$-porous for some $\roo>0$. Furthermore, we may assume that $A$ is
closed since it is clear from the definition that the closure of a uniformly
$\roo$-porous set is uniformly $\roo$-porous. Assume on the contrary that
$\mu(A) > 0$. By Remark \ref{introrems}.(5), we may assume that
$\mu$ is globally doubling. Thus using the density point theorem
\cite[Theorem 1.8]{Heinonen2001}, we choose $x \in A$ for which
\begin{equation}\label{density1}
   \lim_{r \downarrow 0} \frac{\mu\bigl( A \cap B(x,r)
     \bigr)}{\mu\bigl( B(x,r) \bigr)} = 1.
\end{equation}

Since $A$ is uniformly $\roo$-porous we find $0<r_p<r_\mu$ such that for
all $0<r<r_p$ and $\roo'<\roo$ there exists $y \in X$ for which
\begin{equation*}
    B(y,\roo' r) \subset B(x,r) \setminus A.
\end{equation*}
Hence by Lemma \ref{thm:harjoitustehtava},  we get for any $0<r<r_p$
\begin{equation*}
    \frac{\mu\bigl( B(x,r) \setminus A \bigr)}{\mu\bigl( B(x,r)
      \bigr)} \ge \frac{\mu\bigl( B(y,\roo' r) \bigr)}{\mu\bigl( B(x,r)
      \bigr)} \ge c_\mu^{\frac{\log\roo'}{\log 2}-1} > 0,
  \end{equation*}
  contrary to \eqref{density1}.
\end{proof}

\section{Dimension estimates for porous sets}

It is well known that in $\R^n$
\begin{equation}\label{rnestimate}
\dimp(A)\leq n-c\roo^n
\end{equation}
for any $\roo$-porous set $A\subset\R^n$. Here $c$ is a positive constant
depending on $n$. In particular,
$\dimp(A)<n$ for all porous sets $A\subset\R^n$. In this
section we discuss whether these estimates are valid in
$s$-regular metric measure spaces and, more generally, on spaces that carry
a doubling measure. In \cite[Lemma 5.8]{DS} it is stated that in the
$s$-regular case $\dimp(A)\leq s-\eta$, where $\eta$ depends on porosity and 
the constants of the $s$-regular measure. The proof is based on generalized 
dyadic cubes whose
side lengths are powers of $\varrho$. Thus (as in $\mathbb R^n$) this argument
will give that $\eta=c(\log\tfrac 1\varrho)^{-1}\varrho^s$. In
\cite[Lemma 3.12]{BHR} a different method is used to show that even the
Assouad dimension of a porous subset of a globally $s$-regular space is less
than $s$. (Recall that the Assouad dimension is always at least the packing
dimension.) Pushing this argument further we will show that
the factor $(\log\tfrac 1\varrho)^{-1}$ is not needed in the $s$-regular case.
As a tool we need the following generalization of
\eqref{lporr}-\eqref{lporA} which is a modification of the mean
$\varepsilon$-porosity from \cite{KoskelaRohde1997}.

\begin{definition}\label{meanporo}
Let $0<\varrho\le 1$, $D>1$, $0<p\le 1$ and $n_0,k_0\in\mathbb N$.
For all $k\in\mathbb N$ and $x\in X$, we denote by $A_k(x)$ the annulus
\[
A_k(x)=\{y\in X\,:\,D^{-k}<d(x,y)\le D^{-k+1}\}.
\]
Furthermore, for $A\subset X$ define
\[
\psi_k(x)=\begin{cases} 1&\text{ if }A_k(x)\text{ contains }y\text{ with }
                            \dist(y,A)\ge\varrho d(y,x)\\
                        0&\text{ otherwise}.
           \end{cases}
\]
Let
\[
S_{k_0,n}(x)=\sum_{k=k_0+1}^{k_0+n}\psi_k(x).
\]
The set $A\subset X$ is
$(\varrho,D,p,n_0,k_0)$-mean porous if $S_{k_0,n}(x)\ge pn$ for all $x\in A$
and $n\ge n_0$.
\end{definition}

Lemma \ref{meansausage} generalizes the arguments of
\cite[Lemma 2.8]{MartioVuorinen1987}, \cite[Theorem 2.1]{KoskelaRohde1997}
and \cite[Lemma 3.12]{BHR}
to our setting. For the purpose of proving it, we state an auxiliary result
that can be found from \cite[Lemma 4.2]{Bojarski} and for the convenience of 
the reader we prove it. Moreover, we use the notation $\chi_B$ for the 
characteristic function of a set $B\subset X$.

\begin{lemma}\label{heinonenexercise}
Suppose that $\mu$ is a globally doubling measure on $X$, that is,
$r_\mu=\diam(X)$. Let $\{x_k\}_{k\in\N }$ be a collection of points in $X$ and 
let $\{r_k\}_{k\in\N}$ and $\{a_k\}_{k\in\N}$ be sequences of positive real 
numbers. Then for all $R\ge1$ and $1\le q<\infty$
\[
\Vert\sum_{k\in\N}a_k\chi_{B(x_k,Rr_k)}\Vert_{L_\mu^q(X)}
\le C_BR^tq\Vert\sum_{k\in\N} a_k\chi_{B(x_k,r_k)}\Vert_{L_\mu^q(X)},
\]
where $t=\tfrac{\log c_\mu}{\log 2}$ and $C_B$ depends on $c_\mu$ only.
Moreover,  if $\mu$ is $s$-regular, then we may choose $t=s$ and $C_B$ depends 
only on $a_\mu$, $b_\mu$ and $s$.

\end{lemma}

\begin{proof}
A straightforward calculation gives the claim in the case $q=1$.
For the case $q>1$, let $\phi\in L^p(X)$, where $p=q/(q-1)$. Then we have
\begin{align*}
\int_X\vert\phi(x)\sum_{k\in\N}a_k&\chi_{B(x_k,Rr_k)}(x)\vert\,d\mu(x)\\
&\le\sum_{k\in\N}a_k\int_{B(x_k,Rr_k)}\vert\phi(x)\vert\,d\mu(x)\\
&\le\sum_{k\in\N}a_k\mu(B(x_k,Rr_k))\inf_{x\in B(x_k,r_k)}M\phi(x)\\
&\le c_\mu R^t\int_XM\phi(x)\sum_{k\in\N}a_k\chi_{B(x_k,r_k)}\,d\mu(x)\\
&\le c_\mu R^t\left(\int_X (M\phi)^p\,d\mu\right)^{\frac{1}{p}}
 \left(\int_X\big(\sum_{k\in\N}a_k\chi_{B(x_k,r_k)}\big)^q
   \,d\mu\right)^{\frac{1}{q}}\\
&\le c_\mu R^tc_1\left(\frac{p}{p-1}\right)^{\frac 1p}\left(\int_X |\phi|^p
   \,d\mu\right)^{\frac{1}{p}}\left(\int_X\big(\sum_{k\in\N}a_k
   \chi_{B(x_k,r_k)}\big)^q\,d\mu\right)^{\frac{1}{q}},
\end{align*}
where we have used the non-centred Hardy-Littlewood maximal function
\[
M\phi(x)=\sup_{B(y,r)\ni x}\mu(B(y,r))^{-1}\int_{B(y,r)}\vert\phi\vert\,d\mu,
\]
Lemma \ref{thm:harjoitustehtava}, H\"older's inequality and the fact that when 
$p>1$
\[
\int_X (M\phi)^p\,d\mu\le c_1^p \frac{p}{p-1}\int_X\vert\phi\vert^p\,d\mu,
\]
where $c_1$ depends only on $c_\mu$
(see \cite[Theorem 2.2, Remark 2.5]{Heinonen2001}).
Thus the claim follows from the duality of $L^q$-spaces.

In the case $\mu$ is s-regular replace in the proof $c_\mu$ with $b_\mu/a_\mu$ 
and $t$ with $s$
and notice that $c_1$ depends in this case only on $a_\mu$, $b_\mu$ and $s$.
\end{proof}

\begin{lemma}\label{meansausage}
Suppose that $\mu$ is a doubling measure on $X$. Let $x_0\in X$ and 
$0<R<\frac{r_\mu}2$. If $A\subset B(x_0,R)$ is $(\varrho,D,p,n_0,k_0)$-mean 
porous, then there is an absolute constant $C_1$ and a constant $C_2$  that 
depends only on $c_\mu$ so that
\[
\mu(A(r))\le C_1D^{k_0\delta}\mu(A(2D^{-k_0}))r^\delta
\text{ for all } r<D^{-n_0-k_0}
\]
where $\delta=C_2(\log D)^{t-1}D^{-3t}p\varrho^t$ and
$t=\frac{\log c_\mu}{\log 2}$. Moreover, if
$\mu$ is $s$-regular then we may choose $t=s$ and $C_2$
depends only on $a_\mu,b_\mu$ and $s$.
\end{lemma}

\begin{proof}
By restricting our measure $\mu$ to the ball $B(x_0,R)$ we may assume that it 
is globally doubling (or globally $s$-regular). Define
\[
\widetilde{\mathcal B}
=\bigl\{B(z,r_z):z\in A(D^{-k_0})\setminus\overline A\text{ and }
r_z=\frac{\log D}{20D^2}\dist(z,A)\bigr\}.
\]
By the $5r$-covering theorem (see \cite[Theorem 1.2]{Heinonen2001}) we find a 
countable pairwise disjoint subfamily
$\mathcal B$ of $\widetilde{\mathcal B}$ such that
\begin{equation}\label{5rprop1}
A(D^{-k_0})\setminus\overline A
\subset\bigcup_{B(z,r_z)\in\mathcal B}B(z,5r_z).
\end{equation}

Letting $j\in\mathbb N$ and $x\in A(2^{-j})$, choose $x'\in A$ such that
$d(x,x')<2^{-j}$. Assume that there is $k\ge k_0+1$ with $\psi_k(x')=1$. Take
$y\in A_k(x')$ such that $\dist(y,A)\ge\roo d(y,x')$. Using \eqref{5rprop1}
we find $B(z,r_z)\in\mathcal B$ such that $y\in B(z,5r_z)$.
Technical elementary computations involving triangle inequality show that
\begin{equation}\label{1claim}
B(z,5r_z)\subset A_{k+1}(x')\cup A_k(x')\cup  A_{k-1}(x').
\end{equation}
Furthermore, under the extra assumption  $D^{-k}\ge2^{-j}$ similar calculations
give
\begin{equation}\label{claim2}
x\in B\big(z,\frac{75D^3}{\roo\log D}r_z\big).
\end{equation}


Clearly,  $D^{-k}\ge2^{-j}$ provided that $k\le\frac{\log2}{\log D}j$.
Thus if $\psi_k(x')=1$ for $k_0+1\le k\le\frac{\log2}{\log D}j$ we find, by
\eqref{claim2}, a ball $B(x_k,r_{x_k})\in\mathcal B$ such that
$x\in B(x_k,\tfrac{75D^3}{\roo\log D}r_{x_k})$. The fact that $A$ is
$(\varrho,D,p,n_0,k_0)$-mean porous gives
\[
S_{k_0,n}(x')=\sum_{k=k_0+1}^{k_0+n}\psi_k(x')\ge pn\text{ whenever }n\ge n_0.
\]
Letting $j_0>\tfrac{\log D}{\log2}(n_0+k_0)$, we have for
all $j\ge j_0$
\[
\#\{k:k_0+1\le k\le\frac{\log2}{\log D}j\text{ and }\psi_k(x')=1\}
 \ge\frac p2\Bigl(\frac{\log2}{\log D}j-k_0\Bigr).
\]
Combining this with \eqref{1claim} implies that
for all $j\ge j_0$ and $x\in A(2^{-j})$
\begin{equation}\label{ballsum}
\sum_{B(z,r_z)\in\mathcal B}\chi_{B(z,\frac{75D^3}{\roo\log D}r_z)}(x)
\ge\frac p6\Bigl(\frac{\log2}{\log D}j-k_0\Bigr).
\end{equation}
Indeed, for at least $\tfrac p2(\tfrac{\log2}{\log D}j-k_0)$ different $k'$s
we find a ball $B(x_k,r_{x_k})\in\mathcal B$ such that
$x\in B(x_k,\frac{75D^3}{\roo\log D}r_{x_k})$. However, because of 
\eqref{1claim} each $B(x_k,r_{x_k})$ can be taken into account at most three 
times.

We finish the proof by verifying that for all $j\ge j_0$
\begin{equation}\label{finalclaim}
\mu(A(2^{-j}))\le 11D^{k_0\delta}\mu(A(2D^{-k_0}))2^{-j\delta}
\end{equation}
where
\[
\delta=\frac{\log2}{18C_B75^t}(\log D)^{t-1}D^{-3t}p\roo^t
\text{ and }t=\frac{\log c_\mu}{\log2},
\]
and $C_B$ is the constant from Lemma \ref{heinonenexercise}. (If $\mu$ is 
$s$-regular replace here and at the rest of the proof $t$ with $s$.)
Notice that $\delta<1$.
Our claim easily follows from this. For \eqref{finalclaim} it suffices to show
that for all $j\ge j_0$
\begin{equation}\label{absfinalclaim}
\int_{A(2^{-j})}2^{\gamma(\roo\log D)^t(\frac p6k_0+
\sum_{B(z,r_z)\in\mathcal B}\chi_{B(z,\frac{75D^3}{\roo\log D}r_z)}(x))}
  \,d\mu(x)\le11D^{k_0\delta}\mu(A(2D^{-k_0}))
\end{equation}
with $\gamma=(3C_B(75D^3)^t)^{-1}$.
This is so because from \eqref{ballsum} we obtain for all $x\in A(2^{-j})$
that
\[
2^{\gamma(\roo\log D)^t
\sum_{B(z,r_z)\in\mathcal B}\chi_{B(z,\frac{75D^3}{\roo\log D}r_z)}(x))}
\ge2^{\gamma(\roo\log D)^t\frac p6(\frac{\log2}{\log D}j-k_0)}.
\]
Moreover, combining this with \eqref{absfinalclaim} gives
\begin{align*}
\mu(A(2^{-j}))&
=\mu(A(2^{-j}))2^{-j\gamma(\roo\log D)^t\frac p6\frac{\log 2}{\log D}}
2^{j\gamma(\roo\log D)^t\frac p6\frac{\log 2}{\log D}}\\
&=2^{-j\gamma(\roo\log D)^t\frac p6\frac{\log 2}{\log D}}
\int_{A(2^{-j})}
2^{\gamma(\roo\log D)^t\frac p6\frac{\log 2}{\log D}j}\,d\mu\\
&\le11D^{k_0\delta}\mu(A(2D^{-k_0}))2^{-j\gamma(\roo\log D)^t\frac p6
\frac{\log 2}{\log D}},
\end{align*}
and therefore \eqref{finalclaim} is valid.

To prove \eqref{absfinalclaim}, write
\[
u(x)=\gamma(\roo\log D)^t\sum_{B(z,r_z)\in\mathcal B}
     \chi_{B(z,\frac{75D^3}{\roo\log D}r_z)}(x).
\]

Now we obtain from Lemma \ref{heinonenexercise} (used with $q=k$) and from the 
fact that for $B\in\mathcal B$ we have
$B\subset A(2D^{-k_0})$ that
\begin{align*}
\int_{A(2^{-j})}&2^{u(x)}\,d\mu(x)
\le\int_{A(D^{-k_0})}\exp(u(x))\,d\mu(x)\\
&\le\mu(A(D^{-k_0}))+\sum_{k=1}^\infty\frac{(\gamma(\roo\log D)^t)^k}{k!}
\int_X(\sum_{B(z,r_z)\in\mathcal B}
     \chi_{B(z,\frac{75D^3}{\roo\log D}r_z)}(x))^k\,d\mu(x)\\
&\le\mu(A(D^{-k_0}))+\sum_{k=1}^\infty\frac{(\gamma(\roo\log D)^tC_B
  (75D^3)^tk)^k}{k!(\roo\log D)^{tk}}
  \int_X(\sum_{B\in\mathcal B}\chi_B(x))^k\,d\mu(x)\\
&\le\mu(A(2D^{-k_0}))\bigl(1+\sum_{k=1}^\infty\frac{(\gamma C_B
  (75D^3)^tk)^k}{k!}\bigr)\\
&=\mu(A(2D^{-k_0}))\Bigl(1+\sum_{k=1}^\infty\frac1{k!}\Bigl(\frac k3
  \Bigr)^k\Bigr)\\
&\le11\mu(A(2D^{-k_0})),
\end{align*}
where the last inequality follows since
\[
\frac{\frac{1}{(k+1)!}(\frac{k+1}{3})^{k+1}}{\frac{1}{k!}(\frac{k}{3})^k}
=\frac1{3}\bigl(1+\frac1k\bigr)^k\uparrow\frac e3
\]
as $k\to\infty$. Finally,
\begin{align*}
\int_{A(2^{-j})}&2^{\gamma(\roo\log D)^t(\frac p6k_0+
  \sum_{B(z,r_z)\in\mathcal B}\chi_{B(z,\frac{75D^3}{\roo\log D}r_z)}(x))}
  \,d\mu(x)\\
&=2^{\gamma(\varrho\log D)^t\frac p6k_0}
  \int_{A(2^{-j})}2^{u(x)}\,d\mu(x)\\
&\le11\cdot2^{\gamma(\roo\log D)^t\frac p6k_0}\mu(A(2D^{-k_0}))\\
&=11D^{k_0\delta}\mu(A(2D^{-k_0}))
\end{align*}
finishing the proof.
\end{proof}

\begin{remark}\label{uniismean}
(1) Assume that $\por(A,x,r)>\roo$ for all $x\in A$ and $0<r<r_p$.
Let $D>1$. Choose $k_0\in\mathbb N$ such that  $D^{-k_0}<r_p$. Then there is
$y\in X$ such that $B(y,\roo D^{-k_0})\subset B(x,D^{-k_0})\setminus A$ which
in turn implies that $y\in A_k(x)$ for some
$k_0+1\le k\le k_0-\lfloor\frac{\log\varrho}{\log D}\rfloor$, where 
$\lfloor a\rfloor$ is the greatest integer $i$ satisfying $i\le a$. Repeating 
this argument for $k_0-m\lfloor\frac{\log\varrho}{\log D}\rfloor$, 
$m=1,2,\dots$, we see that $A$ is
$(\varrho,D,-\frac 12\lfloor\frac{\log\varrho}{\log D}\rfloor^{-1},
-\lfloor\frac{\log\varrho}{\log D}\rfloor,k_0)$-mean porous.
Note that it is not possible to obtain mean porosity
with $p=1$ or $p$ close to 1 unless one takes $D=\frac 1\varrho$. The
reason for this is that in general metric spaces for fixed $D$ the annuli
$A_k(x)$ may be empty for many $k$'s. However, in $s$-regular spaces one may
find a $D$ independent of $\varrho$ such that all annuli are non-empty as
explained in the next remark.

(2) Assume that $\mu$ is an $s$-regular measure on $X$ and
$\por(A,x,r)>\varrho$ for all $x\in A$ and $0<r<r_p$. Let
$l=\frac12(\frac{a_\mu}{b_\mu})^{\frac 1s}$ and $D=2(1-l)^{-1}l^{-2}$.
Choose $k_0\in\mathbb N$ such that $D^{-k_0}<\min\{r_p,r_\mu\}$. We verify that
$A$ is $(\frac13l^2\roo,D,1,1,k_0)$-mean porous. Consider $0<r\le D^{-k_0}$.
Since $\mu$ is $s$-regular and $r<r_\mu$ there exists
$y\in B(x,lr)\setminus B(x,l^2r)$.
Assuming that there is $z\in B(y,\frac 12l^2r)\cap A$, we find, using uniform
porosity of $A$ and the fact $r<r_p$, $w\in X$ such that
$B(w,\frac12\varrho l^3r)\cap A=\emptyset$ and
$\frac12\roo l^3r+d(w,z)\le\frac12l^3r$. This gives
\[
d(w,x)\ge d(x,y)-d(y,z)-d(z,w)\ge\frac12 l^2(1-l)r=\frac rD
\]
and
\[
d(w,x)\le d(x,y)+d(y,z)+d(z,w)<\frac 32lr<r.
\]
Letting $k\ge k_0+1$ and choosing $r=D^{-(k-1)}$ in the above inequalities
gives $w\in A_k(x)$ for all $k\ge k_0+1$. This combined with the fact that
\[
\dist(w,A)\ge\frac12\varrho l^3r>\frac13l^2\varrho d(w,x)
\]
gives the claim. If $B(y,\frac12l^2r)\cap A=\emptyset$ we may take $w=y$.
\end{remark}

In the following two corollaries we verify scaling properties of measures
of $r$-neighbourhoods of bounded uniformly porous sets for small scales
$r$.

\begin{corollary}\label{localsausage}
Suppose that $\mu$ is a doubling measure on $X$. Let $x_0\in X$ and 
$0<r_0<\frac\roo{12}\min\{r_p,r_\mu\}$. Then there exists a constant $C_3>0$ 
which depends only on $c_\mu$ such that for any uniformly $\roo$-porous 
set $A\subset X$ we have
\[
\mu\big((A\cap B(x_0,r_0))(r)\big)\le C_3 c_\mu^{-\frac{\log\roo}{\log2}}
\mu(B(x_0,r_0))\Big(\frac r{r_0}\Big)^\delta\text{ for all }0<r<r_0,
\]
where $t=\frac{\log c_\mu}{\log2}$ and 
$\delta=c(\log\frac 1\roo)^{-1}\varrho^t$ with $c$ depending only on $c_\mu$.
\end{corollary}

\begin{proof}
Let $0<\roo'<\roo$ be such that $r_0<\tfrac{\roo'}{12}\min\{r_p,r_\mu\}$. 
Choose $k_0$ to be the largest integer so that 
$\tfrac 6{\roo'} r_0\le 2^{-k_0}<\min\{r_p,r_\mu\}$.
By Remark \ref{uniismean}.(1), $A$ is
$(\varrho',2,p,n_0,k_0)$-mean porous for $p=c_0\log(\tfrac 1{\roo'})^{-1}$ and
$n_0=-\lfloor\tfrac{\log\roo'}{\log 2}\rfloor$, where $c_0>0$ is an absolute 
constant. Thus by Lemma
\ref{meansausage} we obtain for all $0<r<2^{-n_0-k_0}$ that
\[
\mu\big((A\cap B(x_0,r_0))(r)\big)\le C_12^{k_0\delta}\mu\big(
   (A\cap B(x_0,r_0))(2^{-k_0+1})\big)r^\delta.
\]

Notice that 
$(A\cap B(x_0,r_0))(2^{-k_0+1})\subset B(x_0,\tfrac{25}{\roo'}r_0)$.
The claim follows by applying the doubling condition and using Lemma 
\ref{thm:harjoitustehtava}, by noting that
$r_0<2^{-n_0-k_0}$, $2^{k_0\delta}<r_0^{-\delta}$
and by letting $\roo'$ tend to $\roo$.
\end{proof}

\begin{corollary}\label{localsausagereg}
Suppose that $\mu$ is an $s$-regular measure on $X$. There exist positive 
constants $D_1$ and $C_4$ which depend only on $a_\mu$, $b_\mu$ and $s$ so 
that for any $x_0\in X$ and $0< r_0< D_1\min\{r_p,r_\mu\}$ we have for any
uniformly $\roo$-porous $A\subset X$ that
\[
\mu\big((A\cap B(x_0,r_0))(r)\big)\le C_4\mu(B(x_0,r_0))
  \Big(\frac r{r_0}\Big)^\delta\text{ for all }0<r<r_0,
\]
where $\delta=c\roo^s$ and $c$ depends only on $a_\mu, b_\mu$ and $s$.
\end{corollary}

\begin{proof}
Let
$l=\frac12(\frac{a_\mu}{b_\mu})^{\frac 1s}$ and $D=2(1-l)^{-1}l^{-2}$.
Set $D_1=\tfrac 12 D^{-2}$. Let $0<\roo'<\roo$. Choose $k_0$ to be the largest 
integer so that $Dr_0\le D^{-k_0}<\min \{r_p,r_\mu\}$.
By Remark \ref{uniismean}.(2), $A$ is
$(\tfrac 13 l^2\varrho',D, 1, 1,k_0)$-mean porous. Thus by Lemma
\ref{meansausage} we obtain for all $0<r<D^{-1-k_0}$ that
\[
\mu\big((A\cap B(x_0,r_0))(r)\big)\le C_1D^{k_0\delta}\mu\big(
  (A\cap B(x_0,r_0))(2D^{-k_0})\big)r^\delta,
\]
where $\delta=c(\roo')^s$ for a constant $c$ that
depends only on $a_\mu, b_\mu$ and $s$.

Notice that $(A\cap B(x_0,r_0))(2D^{-k_0})\subset B(x_0,(1+2D/D_1)r_0)$.
The claim follows by applying the $s$-regularity, by noting that
$r_0<D^{-1-k_0}$ and $D^{k_0\delta}<r_0^{-\delta}$
and by letting $\roo'$ tend to $\roo$.

\end{proof}


Next we prove the analogue of \eqref{rnestimate} for $s$-regular metric spaces.
As mentioned in the Introduction this is asymptotically sharp as $\roo$
tends to zero in $\mathbb R^n$ and thus also in metric spaces.

\begin{theorem}\label{dimestsreg}
Suppose $\mu$ is $s$-regular on $X$. If $A \subset X$ is $\roo$-porous, then
\begin{equation*}
\dimp(A) \le s-c\roo^s.
\end{equation*}
Moreover, if $A$ is uniformly $\roo$-porous and $\diam(A)<r_\mu$, then
\[
\dimm(A)\leq s-c\roo^s.
\]
Here $c$ is as in Corollary \ref{localsausagereg}.
\end{theorem}

\begin{proof}
By Remark \ref{introrems}.(5), for any $\roo'<\roo$ we may represent $A$ as a
countable union of sets $A_{ij}$ with 
$\diam(A_{ij})<D_1\min\{\tfrac 1i,r_\mu\}$ such that $\por(A,x,r)>\roo'$ for 
all $x\in A_{ij}$ and $0<r<\tfrac 1i$, where $D_1$ is as in Corollary 
\ref{localsausagereg}. Moreover, if $A$ is uniformly $\roo$-porous with 
$\diam(A)<r_\mu$,
then it is a finite union of such sets. Thus it is enough to show that
$\dimm(A_{ij})\leq s-\delta$ for all $i$ and $j$ where
$\delta=c(\roo')^s$ as in Corollary \ref{localsausagereg}. Letting 
$x\in A_{ij}$ and using Lemma \ref{thm:content} and Corollary 
\ref{localsausagereg} we get for large  $i$
\begin{gather*}
    \limsup_{r \downarrow 0} M^\lambda(A_{ij},r) \le 2^s a_\mu^{-1}
    \limsup_{r \downarrow 0} \frac{\mu\bigl(A_{ij}(r)\bigr)}{r^{s-\lambda}} \\
    \le 2^s a_\mu^{-1} C \limsup_{r \downarrow 0} \mu\bigl(
      B(x,1/i) \bigr)i^{\delta}r^{\delta+\lambda-s}=0
\end{gather*}
if $\lambda>s-\delta$. Here $C$ is a constant which is independent of $r$.
This gives the claim.
\end{proof}

Theorem \ref{dimestsreg} is not true if we only assume that $\mu$ is
doubling. An easy example is given by defining
$X=\left(\{0\}\cup_{j\in\mathbb{N}}\{2^{-j}\}\right)\times[0,1]$ with the
metric inherited from $\R^2$ and letting $\mu$ be any
doubling measure on $X$. If $N=\{0\}\times[0,1]$, then
$\dimm(N)=\dimp(N)=1=\dimp(X)$. However, it is easy to see that $N$ is
uniformly $\tfrac13$-porous. The following example shows that one can perceive
similar behaviour in geodesic metric spaces as well, even for
maximally porous sets.
Recall that a metric space $(X,d)$ is geodesic if for each pair of points
$x,y\in X$ there exists a path $\gamma \colon [0,1] \to X$ such that
$\gamma(0)=x$, $\gamma(1)=y$, and the length of $\gamma$ is equal to $d(x,y)$.

\begin{example}\label{infinitetree}
We give an example of a complete geodesic doubling metric space
having a subset with maximal dimension and porosity.
The construction is an infinite tree with branches getting smaller and
smaller as we go deeper into the tree. The metric is the natural path metric
induced by the branches (see Figure 1).

Letting
\[
\mathcal N_0=\{\emptyset\}\textrm{ , }
\mathcal{N}_n = \{1, 2\}^n \times ]0,2^{-n}]
\text{ for all }\mathbb N\setminus\{0\}
\textrm{ and } \mathcal{N}_\infty = \{1, 2\}^\mathbb{N}\times\{0\},
\]
define
\[
\mathcal{N} = \mathcal{N}_\infty \cup \bigcup_{n=0}^\infty\mathcal{N}_n.
\]
The metric is given as follows:
For $x\in \mathcal{N}$, let $n(x)\in\mathbb N\cup\{\infty\}$ so that
$x=(x_1,\dots,x_{n(x)},x_{n(x)+1}) \in \mathcal{N}_{n(x)}$.
We denote the distance of $x$ from the root by
\[l(x) = \left\{ \begin{array}{ll}
 1 - 2^{-(n(x)-1)} + x_{n(x)+1} & \textrm{, if $0 < n(x) < \infty$}\\
 1 & \textrm{, if $n(x) = \infty$}\\
 0 & \textrm{, if $n(x) = 0$.}
  \end{array} \right.\]
Given $n\in\mathbb N$, let
\[x|_n = \left\{ \begin{array}{ll}
 (x_1, x_2, \dots, x_n) & \textrm{, if $n(x) \ge n$}\\
 \emptyset & \textrm{, if $n(x)<n$}
  \end{array} \right.\]
be the restriction of $x$.
Define the longest common route $\delta(x,y)$ of a point
$x=(x_1,x_2,\dots)\in\mathcal N$ and
$y=(y_1,y_2,\dots)\in\mathcal N$ from the root by
\[\delta(x,y) = \sup\{m : x|_m = y|_m \ne \emptyset\}\]
and their longest common part $x\wedge y\in\mathcal N_{\delta(x,y)}$
by
\[x\wedge y = \left\{ \begin{array}{lll}
 (x|_{\delta(x,y)}, \min\{&2^{-\delta(x,y)}, n(x)-\delta(x,y)+x_{n(x)+1},&
\textrm{, if $0 < \delta(x,y)<\infty$}\phantom{eee}\\
 &  n(y)-\delta(x,y)+x_{n(y)+1}\})&\\
 \emptyset &  & \textrm{, if $\delta(x,y) = 0$}\\
x & & \textrm{, if $\delta(x,y) = \infty$.}
 \end{array} \right.\]
With these notations we define the metric
$d : \mathcal{N} \times \mathcal{N} \to [0,\infty[$ by
\[d(x,y) = |l(x)-l(x\wedge y)|+ |l(y)-l(x\wedge y)|.\]
Figure \ref{fig:tree} illustrates the metric space $(\mathcal N,d)$ which
is obviously geodesic.
\begin{figure}
\centering
\includegraphics[width=0.4\textwidth]{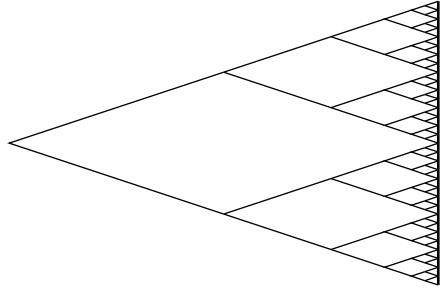}
\caption{An illustration of space $\mathcal{N}$.}
\label{fig:tree}
\end{figure}

We verify  that
\begin{equation}\label{example:dimensions}
\dimm(\mathcal{N}) = \dimh(\mathcal{N}_\infty) = 1
\end{equation}
and
\begin{equation}\label{example:porosity}
\por(\mathcal{N}_\infty) = \frac{1}{2}.
\end{equation}
Indeed, since $\dist(\mathcal N_n,\mathcal N_\infty)=2^{-n}$, the set
$\mathcal N\setminus\cup_{k=1}^n\mathcal N_k$ can be covered by $2^n$
closed balls centred in $\mathcal N_n$ with radii $2^{-n}$. On the other hand,
$\cup_{k=1}^n\mathcal N_k$ can be covered by $n2^n$ balls with radii
$2^{-n}$, and therefore, $\mathcal{N}$
can be covered by $(n+1)2^n$ such balls. Hence
$\dimh(\mathcal{N}_\infty) \le \dimm(\mathcal{N}) \le 1$.
Clearly, $\dimh(\mathcal{N}_\infty)=1$ which gives
\eqref{example:dimensions}.
For \eqref{example:porosity}, take any $x\in \mathcal{N}_\infty$ and $0<r<1$.
Choose $n\in\mathbb{N}$ such that $2^{-n}\le r<2^{-n+1}$. Since
$\dist((x|_n,2^{-n+1}-r),\mathcal N_\infty)=r$, we have for all
$\varepsilon>0$
\[
B((x|_{n},2^{-n+1}-r),(1-\varepsilon)r)
 \subset B(x,2r)\setminus\mathcal{N}_\infty.
\]
This implies \eqref{example:porosity}.

Note that the space $(\mathcal{N}, d)$ is doubling with
a doubling constant $3$, that is, every closed ball with radius $2r$ can be
covered with $3$ closed balls with radius $r$. In particular, it carries a
doubling measure by \cite{LuukkainenSaksman1998}. Moreover, while
$\cup_{k=0}^n\mathcal{N}_k$  is $1$-regular for all $n$ and
$\mathcal{N}_\infty$ is also 1-regular, the set
$\mathcal{N}\setminus\mathcal{N}_\infty$ is not.
\end{example}

By taking a closer look at the above examples one recognizes
that porous sets having maximal dimension are exceptional: there is a
set $N$ with $\mu(N)=0$ so that
$\dimp(A)<\dimp(X)$ for all porous sets $A\subset X\setminus N$. This is not
a coincidence as indicated by the following result.

\begin{theorem}\label{dimestdoub}
Suppose that $\mu$ is a doubling measure on $X$. Then there is $N\subset X$
such that $\mu(N)=0$ and
\[
\dimp(A)\leq\dimp(X)-c(\log\tfrac 1\varrho)^{-1}\varrho^t
\]
for all $\roo$-porous sets $A\subset X\setminus N$ where
$t=\frac{\log c_\mu}{\log 2}$. Here $c$ is as in Corollary \ref{localsausage}.
\end{theorem}

We prove Theorem \ref{dimestdoub} using Corollary \ref{localsausage} and the
following two lemmas.

\begin{lemma}\label{localdimlemma}
Suppose that $\mu$ is a doubling measure on
$X$. Then there is $N\subset X$ with $\mu(N)=0$ such that for all
$s>\dimp(X)$ we have
\[\lim_{r\downarrow 0}\frac{\mu(B(x,r))}{r^s}=\infty\]
 for every $x\in X\setminus N$.
\end{lemma}

\begin{proof}
Choose a decreasing sequence $(s_k)$ such that $s_k\downarrow\dimp(X)$ as
$k\rightarrow\infty$.
We claim that for all $k\in\mathbb N$ there is $N_k\subset X$ with
$\mu(N_k)=0$ such that
\begin{equation}\label{ldinfty}
\lim_{r\downarrow 0}\frac{\mu(B(x,r))}{r^{s_k}}=\infty
\end{equation}
for all $x\in X\setminus N_k$. Fix $k\in\mathbb N$ and suppose to the contrary
that there are $0<D<\infty$ and a
Borel set $A\subset X$ such that $\mu(A)>0$ and
$\liminf_{r\downarrow0}\mu(B(x,r))/r^{s_k}<D$ for all $x\in A$. By
\cite[Theorem 3.16]{Cutler} there exists $C>0$ such that
$\mathcal{P}^{s_k}(A)\geq\frac 1{CD}\mu(A)>0$ which
is impossible since $\dimp(A)\leq\dimp(X)<s_k$. Thus \eqref{ldinfty} is
proved. Defining $N=\cup_{k=0}^\infty N_k$, verifies the claim.
\end{proof}

The following lemma is a substitute for Lemma \ref{thm:content} in
metric spaces that carry a doubling measure.

\begin{lemma}\label{dimlemma1}
Suppose that $\mu$ is a doubling measure on $X$. Let $0<\lambda<\dimp(X)$ and
let $N\subset X$ be as in Lemma \ref{localdimlemma}.
If $A\subset X\setminus N$ has the property that
\[
\limsup_{r\downarrow 0}\frac{\mu(A(r))}{r^\lambda}<\infty,
\]
then $\dimp(A)\leq\dimp(X)-\lambda$.
\end{lemma}

\begin{proof}
Fix $x_0\in X$. Let $s>\dimp(X)$ and define
\[
A_{ki}=\{x\in A\cap B(x_0,i)\,:\,\mu(B(x,r))>r^s\text{ for all }0<r<\frac 1k\}
\]
for all $k$, $i\in\mathbb{N}$. Since $A\subset X\setminus N$ we see
from Lemma \ref{localdimlemma} that
$A=\cup_{k,i}A_{ki}$. It suffices to show that
\begin{equation}\label{dimAk}
\dimm(A_{ki})\leq s-\lambda\text{ for all }k,i\in\mathbb{N}.
\end{equation}
Let $k$, $i\in\mathbb{N}$. Choose $r_0>0$ and $C<\infty$ such that
$\mu(A_{ki}(r))\leq\mu(A(r))\leq Cr^\lambda$
when $0<r<r_0$. If $0<r<\min\{\frac12r_0,\frac 1k\}$, we choose
$\{ B(x_l,r) \}_{l \in I}$ to be a maximal countable collection of mutually
disjoint balls centered at $A_{ki}$. Note that trivially
\begin{equation}\label{Ak5r}
  A_{ki}\subset \bigcup_{l\in I}B(x_l,2r).
\end{equation}
Now
\begin{equation*}
C2^\lambda r^\lambda\ge\mu(A_{ki}(2r))\geq\sum_{l\in I}\mu(B(x_l,r))\ge\#I r^s.
\end{equation*}
This implies that $\#I\le C2^\lambda r^{\lambda-s}<\infty$ which combined
with \eqref{Ak5r} gives
\[
M^{s-\lambda}(A_{ki},2r)\leq \# I (2r)^{s-\lambda} \leq C2^s.
\]
Now
\eqref{dimAk} follows as $r\downarrow 0$.
\end{proof}

\begin{proof}[Proof of Theorem \ref{dimestdoub}]
Let $N$ be as in Lemma \ref{localdimlemma} and let $A\subset X\setminus N$.
Fix $\roo'<\roo$. By Remark \ref{introrems}.(5),
$A$ is a countable union of sets of the form
\[
E=\{x\in A\cap B(x_0,r_0)\,:\,\por(A,x,r)>\roo'\text{ for all }0<r<r_0\}
\]
where $x_0\in X$ and $r_0>0$. Since $\roo'<\roo$ is arbitrary, it
suffices to show that $\dimp(E)\leq\dimp(X)-\delta$ where $\delta$ is
as in Corollary \ref{localsausage} with $\roo$ replaced by $\roo'$. This
follows from Lemma \ref{dimlemma1} since $E\subset X\setminus N$ and
$\limsup_{r\downarrow 0}\mu(E(r))/r^\delta<\infty$ by Corollary
\ref{localsausage}.
\end{proof}

\begin{remark}\label{bestone?}
(1) We do not know if $\dimh(A)\leq \dimh(X)-\delta$ for some $\delta>0$ in
Theorem
\ref{dimestdoub}. Of course, this question is relevant only when
$\dimh(X)<\dimp(X)$.

(2) The essential qualitative difference between Theorems
\ref{dimestsreg} and \ref{dimestdoub} is that in the doubling case we have an
extra factor $(\log\tfrac 1\varrho)^{-1}$ in $\delta$.  This extra factor
appears also in $\mathbb R^n$ if one makes a simple estimate for the Minkowski
dimension using mesh cubes whose side lengths are powers of $\varrho$. To
obtain the optimal upper bound in $\mathbb R^n$ one has to utilize the porosity
at ``all balls'', that is, one has to use mesh cubes whose side lengths are
powers of $D$ for some fixed $D$ independent of $\varrho$.
In our proofs we use annuli and mean porosity instead of mesh cubes. As pointed
out in Remark \ref{uniismean}.(1), the annuli $A_k(x)$ defined using some fixed
$D$ may be empty in general metric spaces for many $k$. Thus we obtain
mean porosity with $p$ containing the factor $(\log\tfrac 1\varrho)^{-1}$.
We do not know whether this factor $(\log\tfrac 1\varrho)^{-1}$ is necessary
or not in Theorem \ref{dimestdoub} although in Remark \ref{uniismean}.(1) it
is which can be seen by considering a disjoint union of Cantor sets
$C_{\roo_i}$ where $\roo_i$ tends to zero.

(3) In Theorem \ref{dimestsreg} the constant $c$ depends on $a_\mu$,
$b_\mu$ and $s$. One may ask whether it is possible that $c$ depends on
$s$ only. However, this is not the case. If in Example \ref{infinitetree}
one chooses $\mathcal N_n=\{1,2\}\times ]0,\lambda^n]$ for $\lambda<\frac 12$,
then the resulting space $\mathcal N$ is 1-regular,
$\por(\mathcal N_\infty)=\frac 12$ and $\dimp(\mathcal N_\infty)\to 1$ as
$\lambda\to\frac 12$. Note that in this case $\frac{b_\mu}{a_\mu}\to\infty$
and thus $c\to 0$ as $\lambda\to\frac 12$.
\end{remark}

\section{Uniform porosity and regular sets}

In this section we prove that if $X$ is $s$-regular and complete, then
uniformly porous and regular subsets of $X$ are closely related in
the following sense: $A\subset X$ is uniformly porous if and only if
there is $0<t<s$ and a $t$-regular set $F\subset X$ such that
$A\subset F$. In $\R^n$ this is well known, see
for example \cite{Salli1991}, \cite[Theorem 5.2]{L},
\cite[Proposition 4.3]{Caetano2002},
\cite[(proof of) Theorem 4.1]{KaenmakiSuomala2004}, and
\cite[Example 6.8]{KaenmakiVilppolainen2006}. We begin by showing that for
each $0<t<s$ there exists a $t$-regular set $F\subset X$. This is a
consequence of the following lemma which was
proven in $\mathbb{R}^n$ in an unpublished Licentiate thesis of
Pirjo Saaranen, see also \cite[Theorem 3.1]{ms}. The proof for
complete $s$-regular spaces is almost the same
but it is included here for the sake of completeness. We denote by $\spt(\mu)$
the support of $\mu$, that is, $\spt(\mu)$ is the smallest closed set $F$ with
$\mu(X\setminus F)=0$.

\begin{lemma}\label{lemma:regular_measures}
Assume that $X$ is complete, $\mu$ is $s$-regular on $X$ and $0<t<s$. Let
$z\in X$ and $0<R<r_\mu$. Then there is a measure $\nu$ with
$\spt(\nu)\subset B(z,2R)$ such that
$\nu(B(z,2R))=R^t$ and
\[
a_\nu r^t\le\nu(B(x,r))\le b_\nu r^t
\]
for all $x\in\spt(\nu)$ and $0<r<R$. Here the constants $a_\nu>0$ and
$b_\nu>0$ depend only on $s$, $t$, $a_\mu$ and $b_\mu$.
\end{lemma}

\begin{proof}
Let $0 < r \le 2^{-2-1/t}R$. Recalling that $X$ is separable, we choose a
maximal countable collection $\{ B(x_i,2^{1+1/t}r) \}_{i \in I}$ of mutually
disjoint balls centred at $B(z,R)$. We may take $x_1=z$. It follows that the
balls $B(x_i,2^{2+1/t}r)$ cover the set $B(z,R)$. Thus
\begin{align*}
\#Ib_\mu(2^{1+1/t}r)^s
\ge & \sum_{i=1}^{\#I}\mu(B(x_i, 2^{1+1/t}r))
     \ge \sum_{i=1}^{\#I}\frac{a_\mu}{2^sb_\mu}\mu(B(x_i, 2^{2+1/t}r))\\
\ge & \frac{a_\mu}{2^sb_\mu}\mu(B(z,R)) \ge \frac{a_\mu^2}{2^sb_\mu}R^s
\end{align*}
by Lemma \ref{thm:harjoitustehtava}, and hence
\begin{equation}\label{ballnumber}
c_1\Big(\frac{R}{r}\Big)^s \le \#I\textrm{ for all } 0<r\le 2^{-2-1/t}R,
\end{equation}
where
\[
c_1 =\frac{a_\mu^2}{2^{2s+s/t}b_\mu^2}.
\]

Choose $d>0$ so small that
\begin{equation}\label{eq:d}
 d<\frac 1{10}2^{-1/t}\textrm{, } d^{s-t}\le\frac{c_1}2
  \textrm{ and } d^t\le\frac 12.
\end{equation}
Taking $r=dR$ in the above process gives disjoint balls
$B(x_i,2^{1+1/t}dR)$ with $x_i\in B(z,R)$. Moreover, by
\eqref{ballnumber} we get the following estimate for the number $m$
of such balls
\begin{equation}\label{eq:m}
\frac{c_1}{d^s}\le m.
\end{equation}
Let $M \in \mathbb{N}$ be such that
\begin{equation}\label{eq:M}
d^{-t}-\frac{1}{2} \le M < d^{-t} + \frac{1}{2}.
\end{equation}
Because
\[
d^{-t}+\frac{1}{2}\le\frac{c_1}{2d^s} + \frac{1}{2}< m
\]
by \eqref{eq:d} and \eqref{eq:m}, we may take $M$ balls from the collection
of disjoint balls
$\{B(x_i,2^{1+1/t}dR) : i = 1, \dots, m\}$. Having roughly advanced
towards our $t$-regular measure on this scale by taking suitable balls,
we proceed by adjusting the radii of the balls to get exactly the regularity
we want.

Fix $d_1 < 1$ so that $d_1^t = M^{-1}$.
Inequalities \eqref{eq:M} and \eqref{eq:d} combine to give
\begin{equation}\label{addlabel}
d_1^t\le \frac{2d^t}{2-d^t}\le\frac 43 d^t<2d^t
\end{equation}
which implies that the balls $B_i = B(x_i, 2d_1R)$, $i = 1, \dots, M$
are disjoint.
Next we repeat the process by taking $B(x_i, d_1R)$ as $B(z,R)$.
In this manner we get disjoint balls
$B(x_{ij}, 2^{1+1/t}dd_1R)$, $j=1,\dots,m_{2i}$, where
$x_{ij}\in B(x_i, d_1R)$, $x_{i1}=x_i$ and
\[
\frac{c_1}{d^s}\le m_{2i}.
\]
For all $i=1, \dots, M$ choose $M$ balls from these collections and
adjust the radii to be $d_1^2R$. Now the balls $B_{ij} = B(x_{ij},2d_1^2R)$
are disjoint and $B_{ij}\subset B_i$ for all $i,j=1,\dots,M$, because
by \eqref{addlabel}
\[
d_1R+2d_1^2R < d_1R+2^{1+1/t}d_1dR = d_1R(1+2^{1+1/t}d)
<2d_1R.
\]

We continue this process. At step $k$ we obtain for all sequences
$(i_1,\dots,i_{k-1})\in\{1, \dots, M\}^{k-1}$ disjoint balls
$B_{i_1\dots i_k}$, $i_k=1,\dots,M$, with centres in the ball
$B(x_{i_1\dots i_{k-1}}, d_1^{k-1}R)$ and with radii $2d_1^kR$ such that
\[
B_{i_1\dots i_k}\subset B_{i_1\dots i_{k-1}}.
\]
Furthermore, $x_{i_1\dots i_{k-1}1}=x_{i_1\dots i_{k-1}}$.
Set
\begin{equation}\label{eq:measure}
\tilde\nu(i_1\dots i_k) = M^{-k}
\end{equation}
for all $k\in\mathbb N$ and $(i_1,\dots,i_k)\in\{1,\dots,M\}^k$. Since $X$ is
$s$-regular, it follows that $B(z,2R)$ is totally bounded and hence also
compact by the completeness of $X$. Thus all closed balls in $B(z,2R)$ are
compact and the set
\begin{equation*}
  E = \bigcap_{k=1}^\infty \bigcup_{i_1,\ldots,i_k} B_{i_1\dots i_k}
  \subset B(z,2R)
\end{equation*}
is nonempty and compact. By the disjointness of the balls $B_{i_1\dots i_k}$,
we may identify the sequence $(i_1,\dots,i_k)$ with $E\cap B_{i_1\dots i_k}$. 
It follows from \eqref{eq:measure} that $\tilde\nu$ is a probability measure 
on the algebra generated by the sets $E\cap B_{i_1\dots i_k}$. Hence, by 
the Carath\'eodory-Hahn extension theorem \cite[Theorem 11.19]{WZ}, 
$\tilde\nu$ extends to a Borel probability measure on $E$.

Next we prove that $\tilde\nu$ is $t$-regular. Take any $x\in\spt(\tilde\nu)$
and $0<r<(1-2d_1)R$. Let $l\in \mathbb{N}$ be such that
\begin{equation}\label{eq:radii}
(1-2d_1)d_1^{l+1}R\le r <(1-2d_1)d_1^lR.
\end{equation}
By \eqref{addlabel} and \eqref{eq:d} we have $d_1<\frac 1{10}$. Thus
inequalities \eqref{eq:radii} guarantee that
\[
B_{i_1\dots i_{l+2}} \subset B(x,r)\subset B_{i_1\dots i_l}
\]
for some $(i_1,\dots,i_{l+2})\in\{1,\dots,M\}^{l+2}$.
This in turn implies with \eqref{eq:measure} and \eqref{eq:radii} that
\[
\frac{d_1^{2t}}{(1-2d_1)^tR^t}r^t\le d_1^{(l+2)t}=M^{-(l+2)}
\le\tilde\nu(B(x,r))\le M^{-l}=d_1^{lt}\le \frac{1}{(1-2d_1)^tR^td_1^t}r^t.
\]
Finally, defining $\nu=R^t\tilde\nu$ gives the desired measure.
\end{proof}

As an immediate consequence we obtain:

\begin{corollary}\label{corollary:regular_subsets}
Assume that $X$ is complete and $\mu$ is $s$-regular on $X$. Then for all
$0<t<s$ there is $F\subset X$ which is $t$-regular.
\end{corollary}

Now we are ready to state the main theorem of this section.

\begin{theorem}\label{treghuok}
Suppose that $\mu$ is $s$-regular on $X$. If $0<t<s$ and $A \subset X$ is
$t$-regular, then $A$ is uniformly $\roo$-porous for some $\roo>0$. Conversely,
if $X$ is complete and $A\subset X$ is uniformly $\varrho$-porous, then for all
$s-c\varrho^s<t<s$ there exists a $t$-regular set $F\subset X$ so that 
$A\subset F$. Here $c$ is as in Corollary \ref{localsausagereg}.
\end{theorem}

\begin{proof}
The first part is proven in \cite[Lemma 3.12]{BHR}. We reprove it here to
obtain a quantitative estimate which in a sense is optimal (see Remark
\ref{finalremark}).
Let $\nu$ be a $t$-regular measure on $A$. Pick $x \in A$ and
$0<r<\frac12\min\{r_\mu,r_\nu\}$, and take $k \in\N$. For each $k \in \N$, we 
choose a maximal countable collection $\{ B(x_i,2^{-k-1}r) \}_{i \in I_k}$ of 
mutually disjoint balls centred at $B(x,r)$. It follows that the balls 
$B(x_i,2^{-k}r)$ cover the set $B(x,r)$.

From the $s$-regularity of $\mu$ we get
\begin{equation*}
a_\mu r^s \le \mu\bigl( B(x,r) \bigr) \le \sum_{i \in I_k}
\mu\bigl( B(x_i, 2^{-k}r) \bigr) \le \# I_k b_\mu 2^{-ks} r^s,
\end{equation*}
and therefore,
\begin{equation} \label{eq:lower_bound_for_I}
\# I_k \ge \frac{a_\mu}{b_\mu}2^{ks}.
\end{equation}
Taking any set $J_k \subset I_k$ for which
$B(x_j,2^{-k-2}r) \cap A \ne \emptyset$ as $j \in J_k$ and
using the fact that $\nu$ is $t$-regular on $A$ and $r<r_\nu$,
we have
\begin{align*}
b_\nu 2^t r^t &\ge \nu\bigl( B(x,2r) \bigr) \ge \sum_{j \in J_k}
\nu\bigl( B(x_j,2^{-k-1}r) \bigr) \\
&\ge \sum_{j \in J_k} \nu\bigl( B(z_j,2^{-k-2}r) \bigr)
\ge \# J_k a_\nu 2^{-(k+2)t} r^t,
\end{align*}
where $z_j \in B(x_j,2^{-k-2}r) \cap A$ as $j \in J_k$. Thus
\begin{equation*}
\# J_k \le \frac{b_\nu}{a_\nu}2^{(k+3)t}.
\end{equation*}
This upper bound is strictly smaller than the lower bound in
\eqref{eq:lower_bound_for_I} when
\begin{equation*}
k>\frac{\log\bigl(\frac{b_\mu b_\nu}{a_\mu a_\nu}2^{3t}\bigr)}{(s-t)\log 2}
  =:K(\mu,\nu).
\end{equation*}
Choosing $k > K(\mu,\nu)$, gives $I_k\setminus J_k\ne\emptyset$, and we find
$i_0 \in I_k$ such that $B(x_{i_0},2^{-k-2}r) \cap A = \emptyset$.
It follows that
\begin{equation}\label{lowporoest}
\por^*(A,x,2r) \ge 2^{-k-3}
\end{equation}
whenever $x\in A$ and $0<r<\frac12\min\{r_\mu,r_\nu\}$. The claim follows from
Remark \ref{introrems}.(2).

Now we prove the opposite direction.
The idea is to use Lemma \ref{lemma:regular_measures} to build
a regular measure inside the voids of suitable reference balls.
Let $0<\roo'<\roo$ and $r_p>0$ be such that
\begin{equation}\label{poroalpha}
\por(A,x,r)>\varrho'
\end{equation}
for all $x\in A$ and $0<r<r_p$. Set $\gamma =\frac{\varrho'}{5}$ and fix 
$n_0\in \mathbb{N}$ such that
$\gamma^{n_0} <D_1\min\{r_p,r_\mu\}$ (see Corollary \ref{localsausagereg}).
From Corollary \ref{localsausagereg} we see that for all $x\in X$ and
$0<r_0\le\gamma^{n_0}$
\begin{equation}\label{eq:enkeksilabelia}
\mu((A\cap B(x,r_0))(r))\le C_4\mu(B(x,r_0))\big(\frac r{r_0}\big)^\delta
\text{ for all }0<r<r_0
\end{equation}
where $\delta=c(\varrho')^s$.

Consider $s-\delta<t<s$. For all $j\in\mathbb N$, we choose a maximal
collection of disjoint balls
$\{B(x_{ji},\gamma^{n_0+j})\}$ so that $x_{ji}\in A$
for all $i\in\mathbb N$. Now we clearly have
\begin{equation} \label{Asubset}
A \subset\bigcup_i B(x_{ji},2\gamma^{n_0+j}).
\end{equation}
For each $i\in\mathbb N$ we find, using inequality \eqref{poroalpha} and the
fact $r<r_p$, points $z_{ji}\in B(x_{ji},\gamma^{n_0+j})$ so that
$B(z_{ji},\varrho'\gamma^{n_0+j})\subset B(x_{ji},\gamma^{n_0+j})\setminus A$.
Define
\[
B_{ji}=B(z_{ji},\tfrac{\varrho'}5\gamma^{n_0+j})=B(z_{ji},\gamma^{n_0+j+1})
\text{ and }2B_{ji}=B(z_{ji},2\gamma^{n_0+j+1}).
\]
Then obviously
$2B_{ji}\subset B(z_{ji},\varrho'\gamma^{n_0+j})
  \subset B(x_{ji},\gamma^{n_0+j})\setminus A$.
Lemma \ref{lemma:regular_measures} implies that for all
$i\in \mathbb N$ there is
a $t$-regular measure $\nu_{ji}$ on $\spt(\nu_{ji})\subset 2B_{ji}$ with
$\nu_{ji}(2B_{ji})=\gamma^{(n_0+j+1)t}$.
Note that the constants $a_{\nu_{ji}}$ and $b_{\nu_{ji}}$ are the same for all
$j$ and $i$, say $a_{\nu_{ji}}=a$ and $b_{\nu_{ji}}=b$. Moreover, it is evident
by the properties of $\nu_{ji}$  that we may choose
$r_{\nu_{ji}}=3\gamma^{n_0+j}$ by adjusting $a$ and $b$. We conclude that
for all $x\in\spt(\nu_{ji})$ and $0<r<3\gamma^{n_0+j}$
\begin{equation}\label{nuijregular}
ar^t\le\nu_{ji}(B(x,r))\le br^t.
\end{equation}

Setting
\[
F=\bigcup_{j,i}\spt(\nu_{ji})\cup A
\]
and
\[
\nu=\sum_{j,i} \nu_{ji},
\]
we clearly have $A\subset F$ and $\nu(X\setminus F)=0$,
and therefore, it suffices to
prove that $\nu$ is $t$-regular on $F$.

We first verify that $\nu$ is $t$-regular on $A$ by showing that there are
constants
$C_1$ and $C_2$ such that for all $x\in A$ and $0<r<\frac12\gamma^{n_0+1}$
we have
\begin{equation}\label{goalA}
C_1r^t\le\nu(B(x,r))\le C_2r^t.
\end{equation}
For the purpose of proving \eqref{goalA}, fix $k\in\mathbb N$ so that
\[
\gamma^{n_0+k+1}\le r < \gamma^{n_0+k}.
\]
From \eqref{Asubset} it follows easily that $2B_{k+3,i}\subset B(x,r)$ for some
$i$ giving
\begin{equation*}
\nu(B(x,r))\ge\nu_{k+3,i}(2B_{k+3,i})=\gamma^{(n_0+k+4)t}\ge\gamma^{4t}r^t.
\end{equation*}
Thus we may choose $C_1=\gamma^{4t}$ in \eqref{goalA}.

For the remaining inequality in \eqref{goalA}, denote by $N_j$ the number of
balls $2B_{ji}$ that intersect $B(x,r)$. Assuming that $j\ge k$, we obtain
\begin{equation}\label{laajennus}
\bigcup_{2B_{ji}\cap B(x,r)\ne\emptyset}2B_{ji}
\subset (A\cap B(x,\frac2{\gamma}r))(\gamma^{n_0+j}).
\end{equation}
Indeed, if $2B_{ji}\cap B(x,r)\ne\emptyset$, then also
$B(x_{ji},\gamma^{n_0+j})\cap B(x,r)\ne\emptyset$ giving
$d(x_{ji},x)\le r+\gamma^{n_0+j}\le r+\frac 1{\gamma}r\le\frac 2{\gamma}r$.
Hence \eqref{laajennus} is valid since $x_{ji}\in A$.
Now \eqref{laajennus} implies with \eqref{eq:enkeksilabelia} that
\begin{align*}
N_j a_\mu 2^s\gamma^{(n_0+j+1)s}
&\le\mu\bigl(\bigcup_{2B_{ji}\cap B(x,r)\ne\emptyset}2B_{ji}\bigr)
  \le \mu((A\cap B(x,\frac2{\gamma}r))(\gamma^{n_0 + j}))\\
&\le C_4\mu(B(x,\frac2{\gamma} r))
  \big(\frac{\gamma^{n_0 + j}}{\frac2{\gamma}r}\big)^{\delta}
  \le C_4 b_\mu\big(\frac2{\gamma}r\big)^{s-\delta}\gamma^{(n_0+j)\delta}\\
&\le  C_4 b_\mu\big(\frac2{\gamma}\big)^{s-\delta}
\gamma^{(n_0+k)s}\gamma^{(j-k)\delta},
\end{align*}
and therefore,
\begin{equation}\label{Nestimate}
N_j \le c_3\gamma^{(k-j)(s-\delta)}\text{ for all }j\ge k,
\end{equation}
where $c_3$ is a constant independent of $k$ and $j$.
Note that $N_j=0$ provided that $j\le k-1$. This is true
because
\[
\dist(2B_{ji},A)\ge 3\gamma^{n_0+j+1}\ge 3\gamma^{n_0+k}>r.
\]
From \eqref{Nestimate} we obtain
\begin{equation*}
\begin{split}
\nu(B(x,r))&=\sum_{j,i}\nu_{ji}(B(x,r)\cap 2B_{ji})
\le\sum_{j=k}^\infty N_j \gamma^{(n_0+j+1)t}\\
&\le c_3\gamma^{t(n_0+1)+k(s-\delta)}
\sum_{j=k}^\infty \gamma^{j(\delta+t-s)}\\
&=\frac{c_3}{1-\gamma^{\delta+t-s}}\gamma^{t(n_0+k+1)}
\le C_2r^t
\end{split}
\end{equation*}
where $C_2=\frac{c_3}{1-\gamma^{\delta+t-s}}$. This completes the verification
of \eqref{goalA}.

We finish the proof by showing that $\nu$ is $t$-regular on $\spt(\nu_{ji})$
for all $j$ and $i$. Fix $j$ and $i$ and let
$y\in\spt(\nu_{ji})$. We first derive the
lower bound in the definition of $t$-regularity.
If $0<r\leq 3\gamma^{n_0+j}$, then
$\nu(B(y,r))\geq\nu_{ji}(B(y,r))\geq a r^t$
by \eqref{nuijregular}. On the other hand, if
$3\gamma^{n_0+j}< r<\gamma^{n_0+1}$, we get
$B(x_{ji},\frac r2)\subset B(y,r)$. Applying \eqref{goalA} implies that
$\nu(B(y,r))\geq\nu(B(x_{ji},\frac r2))\geq C_12^{-t}r^t$.

For the upper regularity bound we first assume that
$0<r\leq\gamma^{n_0+j+1}$. Now $B(y,r)\cap 2B_{j'i'}=\emptyset$ for all
$(j',i')\neq (j,i)$. Indeed, if $j=j'$ this follows from
$B(y,r)\subset B(x_{ji},\gamma^{n_0+j})$ and
$2B_{ji'}\subset B(x_{ji'},\gamma^{n_0+j})$. In the case $j'<j$ assume that
$w\in B(y,r)$. Then $d(w,A)\le d(w,x_{ji})\le 2\gamma^{n_0+j}$. Further, if
$w\in 2B_{j'i'}$, then
$d(w,A)\ge\frac{3\roo'}5\gamma^{n_0+j'}\ge 3\gamma^{n_0+j}$.
The final case $j'>j$ is similar. Thus
$\nu(B(y,r))=\nu_{ji}(B(y,r))\leq b r^t$ by \eqref{nuijregular}.
Finally, supposing that
$\gamma^{n_0+j+1}<r<\tfrac{\varrho'}{12}\gamma^{n_0+1}$ gives
$B(y,r)\subset B(x_{ji},\frac{6r}{\varrho'})$.
Hence $\nu(B(y,r))\leq\nu(B(x_{ji},\tfrac{6r}{\varrho'}))
\leq C_2\big(\tfrac6{\varrho'}\big)^tr^t$
by \eqref{goalA}.
\end{proof}

\begin{remark}\label{finalremark}
(1) Observe that in \eqref{lowporoest} the porosity $\roo$ is proportional
to $\bigl(\frac{a_\nu}{b_\nu}\bigr)^{\frac 1{s-t}}$. The following simple
construction shows that this is sharp. Let $0<t<1$ and choose from $[0,1]$
$N$ evenly distributed intervals of length $N^{-\frac 1t}$. Repeat this
construction and let $\nu$ be the natural measure on the resulting Cantor
set $A$. Then $A$ is $t$-regular, $\por(A)\approx\frac{1-N^{1-\frac 1t}}N$
and $\frac{a_\nu}{b_\nu}=\frac{N^{t-1}}{(1-N^{1-\frac 1t})^t}$. As $N$ tends to
infinity,
$\por(A)\approx\frac 1N\approx\bigl(\frac{a_\nu}{b_\nu}\bigr)^{\frac 1{1-t}}$.
We do not know what is the best asymptotic behaviour as $\frac{a_\nu}{b_\nu}$
is fixed and $t\to s$.

(2) We do not know whether the completeness is needed in Theorem
\ref{treghuok}. Our method does not work without the completeness since in 
that case the set $E$ constructed in the proof of Lemma 
\ref{lemma:regular_measures} may be empty.
\end{remark}

\end{document}